\documentclass[12pt, reqno]{amsart}

\usepackage{amsmath}
\usepackage{amsthm}
\usepackage{amssymb}
\usepackage{latexsym}
\usepackage{graphicx}
\usepackage{anysize}
\usepackage{tikz,pgflibraryplotmarks}
\usepackage{ifthen}
\usepackage{pstricks}
\usepackage{comment}

\usepackage{subfigure}
\usepackage{verbatim}
\usetikzlibrary{shapes,snakes}

\theoremstyle{plain}                    

\theoremstyle{plain}
\newtheorem{theorem}{Theorem}[section]
\newtheorem{lemma}[theorem]{Lemma}

\newtheorem{bigtheorem}{Theorem}[section]

\theoremstyle{definition}

\newtheorem{definition}[theorem]{Definition}
\newtheorem{alg}[]{Algorithm}

\theoremstyle{remark}
\newtheorem{remark}[theorem]{Remark}

\numberwithin{equation}{section}

\numberwithin{equation}{section}

\def \scale {\text{scale}}

\def\P{\mathbb{P}}

\def\R{\mathbb{R}}

\def\B{\mathcal{B}_{\epsilon}}

\def\tg{G}

\newdimen\HilbertLastX
\newdimen\HilbertLastY
\newcounter{HilbertOrder}

\def\DrawToNext#1#2{
   \advance \HilbertLastX by #2
   \advance \HilbertLastY by #1
   \pgfpathlineto{\pgfqpoint{\HilbertLastX}{\HilbertLastY}}
}

\def\Hilbert[#1,#2,#3,#4,#5,#6,#7,#8] {
  \ifnum\value{HilbertOrder} > 0%
     \addtocounter{HilbertOrder}{-1}
     \Hilbert[#5,#6,#7,#8,#1,#2,#3,#4]
     \DrawToNext {#1} {#2}
     \Hilbert[#1,#2,#3,#4,#5,#6,#7,#8]
     \DrawToNext {#5} {#6}
     \Hilbert[#1,#2,#3,#4,#5,#6,#7,#8]
     \DrawToNext {#3} {#4}
     \Hilbert[#7,#8,#5,#6,#3,#4,#1,#2]
     \addtocounter{HilbertOrder}{1}
  \fi
}

\def\hilbert((#1,#2),#3,#4){%
\advance \HilbertLastX by #1
   \advance \HilbertLastY by #2
   \pgfpathmoveto{\pgfqpoint{\HilbertLastX}{\HilbertLastY}}
   \setcounter{HilbertOrder}{#3}
   \Hilbert[1mm,0mm,-1mm,0mm,0mm,1mm,0mm,-1mm]
   \pgfusepath{stroke};
   \draw[step=1mm, gray, very thin] (0mm,0mm) grid (#4,#4);
   }

\def\hilbertnew((#1,#2),#3,#4){%
 \advance \HilbertLastX by #1
   \advance \HilbertLastY by #2
   \pgfpathmoveto{\pgfqpoint{\HilbertLastX}{\HilbertLastY}}
   \setcounter{HilbertOrder}{#3}
   \Hilbert[1mm,0mm,-1mm,0mm,0mm,1mm,0mm,-1mm]
   \pgfusepath{stroke};
   \draw[step=1mm, gray, very thin] (0mm,0mm) grid (#4,#4);

   }

\begin{document}

\baselineskip=1.3\baselineskip

\title[]{\bf Brownian motion with variable drift\\ can be space-filling}

\author[]{Ton\'{c}i Antunovi\'{c}, Yuval Peres and Brigitta Vermesi
}
\address{Ton\'{c}i Antunovi\'{c}\\
University of California, Berkeley\\
Department of Mathematics\\
Berkeley, CA 94720\\
}
\email{tantun@math.berkeley.edu}

\address{Yuval Peres\\
Microsoft Research\\
Theory Group\\
Redmond, WA 98052}
\email{peres@microsoft.com}

\address{Brigitta Vermesi\\
University of Washington\\
Department of Mathematics\\
Seattle, WA 98195}
\email{bvermesi@math.washington.edu}



\begin{abstract}
\end{abstract}

\begin{abstract}
For $d \geq 2$ let $B$ be standard $d$-dimensional Brownian motion. For any $\alpha < 1/d$ we construct an $\alpha$-H\"{o}lder continuous function $f \colon [0,1] \to \mathbb{R}^d$ so that the range of $B-f$ covers an open set. This strengthens a result of Graversen (1982) and answers a question of Le~Gall (1988).
\end{abstract}

\maketitle

\section{Introduction}
For $d\ge 2$, let $B_t$ be a $d$-dimensional standard Brownian
motion and $f:[0,1]\to \R^d$ a continuous function. We say $f$
is \emph{polar} for $d$-dimensional Brownian motion if, for all
$x$,
\[\P_x\{ B_t=f(t) \mbox{ for some } t>0\}=0.\]
It is well known that $B_t$ does not hit points (see Corollary 2.24 in \cite{PM})
almost surely, hence constant functions are polar.
The Cameron-Martin theorem (see Theorem 1.38 in \cite{PM} or Theorem 2.2 in Chapter 8 in \cite{RY}) tells us that for all functions $f$ in the
Dirichlet space $\textbf{D}[0,1]$ (integrals of functions in $\textbf{L}^2[0,1]$), the path $B_t-f(t)$ has the same almost sure properties as the Brownian motion path. Hence functions in $\textbf{D}[0,1]$ are polar.
For dimension two Graversen proved the
following.
\begin{bigtheorem}[Graversen \cite{Grav}]
For all $0<\gamma<1/2$, there exists a $\gamma$-H\"{o}lder
continuous function $f:\mathbb{R}^+\to \R^2$
which is non-polar
for planar Brownian motion.
\end{bigtheorem}
Graversen also conjectured that all $1/2$-H\"{o}lder continuous
functions are polar and gave a partial answer to this
conjecture. In \cite{LeGall}, Le~Gall proved Graversen's
conjecture and gave a similar result for dimensions $d\ge 3$.
\begin{bigtheorem}[Le~Gall \cite{LeGall}]\label{thm:LG}
Let $f:\mathbb{R}^+\to \R^d$ be a continuous function. Then $f$ is
polar for $d$-dimensional Brownian motion if it satisfies one of
the following conditions:
\begin{itemize}
\item[(i)] $d=2$ and for all $K>0$ there exists a $\delta>0$
such that for all $s, t$ with $0 \leq s < t \leq K$ and $0<t-s<\delta$,
\begin{equation}\label{2-condition}
|f(t)-f(s)| \le \left[ 2(t-s)\log\log\left(\frac{1}{t-s}\right)
\right]^{1/2},
\end{equation}
\item[(ii)] $d\ge 3$ and for each $K>0$,
\begin{equation}\label{d condition}
\lim_{\delta\to 0}\, \sup_{\substack{0\le s< t\le
K\\t-s<\delta}}\, \frac{|f(t)-f(s)|}{(t-s)^{1/d}}=0.
\end{equation}
\end{itemize}
\end{bigtheorem}

At the end of \cite{LeGall} Le~Gall asked whether for each $\gamma < 1/d$ there exist $\gamma$-H\"{o}lder continuous functions which are non-polar for $d$-dimensional Brownian motion.
By a simple application of Fubini's theorem, a function $f$ is non-polar if and only if the expected volume of the range of $B-f$ is positive.
In the following theorem we give a positive answer to
Le~Gall's question and also strengthen Graversen's Theorem.
\begin{theorem}\label{thm:main}
For $d\ge 2$, let $B_t$ be a $d$-dimensional Brownian
motion. Then for any $\gamma<1/d$, there exists a
$\gamma$-H\"{o}lder continuous function $f:[0,1]\to\R^d$ for
which the range of $B_t-f(t)$, for $t \in [0,1]$, covers an open set almost surely.
\end{theorem}

Note that, in Theorem \ref{thm:LG}, the condition for polarity
in higher dimensions is stronger than for 2 dimensions. We prove
the following theorem which shows that a stronger condition is
indeed necessary.
\begin{theorem}\label{thm:standard Hilbert}
For $d\ge 3$, there exists a $1/d$-H\"{o}lder continuous
function $f:[0,1]\to \R^d$ that is non-polar for $d$-dimensional
Brownian motion.
\end{theorem}
In the proof of Theorem \ref{thm:standard Hilbert} we take $f$
to be the standard $d$-dimensional Hilbert curve.

\section{Hilbert curves}\label{sec:heuristics}
To prove Theorem \ref{thm:main} we will construct space filling
curves that remain space filling even after being perturbed by
Brownian motion. These curves are modifications of the standard
Hilbert curve construction which is briefly described below, for
dimension two.

\subsection{The standard Hilbert curve}
Start with the unit square $[0,1]^2$. In the first iteration,
subdivide the unit square into four sub-squares of side length
$1/2$ and join their centers via three line segments. Choose a
direction for the resulting path. In iteration $n+1$, take the
path from iteration $n$, rescaled by $1/2$, and put one copy in
each of the four sub-squares, rotated clockwise by $\pi/2$, $0$,
$0$, $3\pi/2$ in the bottom-left, top-left, top-right and
bottom-right sub-square, respectively. Then connect the four
copies by adding three connection segments. The limiting curve
obtained by following this iterative construction is called the
Hilbert curve.
Figure \ref{fig:Hilbert} illustrates the first
four steps in the standard Hilbert curve construction, where we
always traverse the path starting at the left bottom square and
ending in the right bottom square.


\begin{figure}[h]
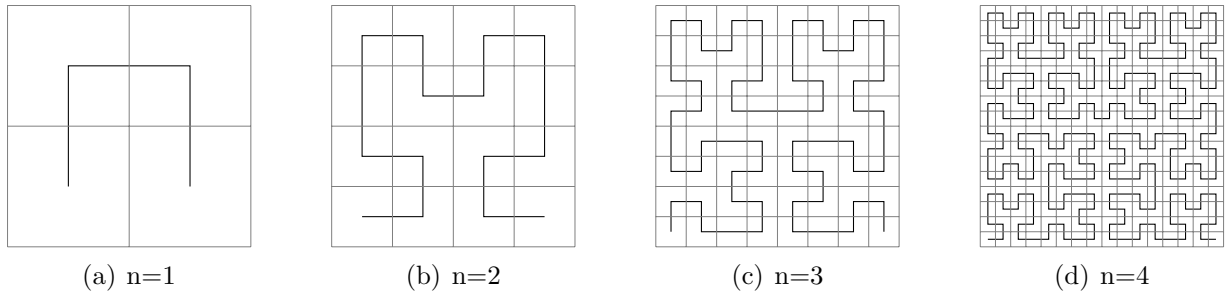

  \begin{center}
    \subfigure[n=1]{\tikz[scale=16]\hilbert((0.5mm,0.5mm),1,2mm);}
    \hfill
    \subfigure[n=2]{\tikz[scale=8]\hilbert((0.5mm,0.5mm),2,4mm);}
    \hfill
   \subfigure[n=3]{\tikz[scale=4] \hilbert((0.5mm,0.5mm),3,8mm);}
    \hfill
   \subfigure[n=4]{\tikz[scale=2] \hilbert((0.5mm,0.5mm),4,16mm);}
 \end{center}
  \caption{The first 4 steps in the Hilbert curve construction}
  \label{fig:Hilbert}
\end{figure}


An equivalent description of the Hilbert curve is as follows.
For $n\ge 1$, at iteration $n,$ the unit square is partitioned
into a collection of $4^n$ squares of side length $2^{-n}$. Each of the $4^n$ squares
are labeled by some string $a_1,\dots,a_n$ with entries in
$\{0,1,2,3\}$, which determines the order in which the squares are traversed. The main idea is to find an algorithm that
recursively determines the order in which sub-squares are
traversed given the order of their ''ancestor'' squares. Not
picking the right order causes the limiting function not to be
continuous. The requirement is that for any sequence $a_1 \dots
a_n$ of elements in $\{0,1,2,3\}$ with $a_n \in \{0,1,2\}$, the
square labeled by $a_1\dots a_n 33 \dots 3$ is very close to the
square $a_1\dots (a_n+1) 00 \dots 0$.

\begin{definition}\label{def: traversal order}
The \textit{traversal order} will be described by functions $h$
satisfying:
\begin{itemize}
\item [(i)] $h:\{0,1,2,3\}\to\{(-1,-1),(1,-1),(1,1),(1,-1)\}$ is
one-to-one,
\item [(ii)] $h(\{0,2\})=\{(-1,-1),(1,1)\}$,
\item [(iii)] $h(\{1,3\})=\{(1,-1),(-1,1)\}$.
\end{itemize}
\end{definition}
Suppose the string $a_1,\dots,
a_n$ denotes the labeling of a square in $\mathcal{F}_n$. The
traversal order of its four sub-squares is given by the function
$f_{a_1,\dots, a_n}$, which we now construct
recursively.
For the empty string $\emptyset$ define $f_\emptyset$ as
\begin{equation}\label{eq: starting function}
(f_\emptyset(0),f_\emptyset(1),f_\emptyset(2),f_\emptyset(3)) =
((-1,-1),(-1,1),(1,1),(1,-1)).
\end{equation}
\begin{alg}[Standard Hilbert]\label{alg:main}
Given $f_{a_1,\dots, a_{n-1}}$ and $a_n$ define $f_{a_1, \dots, a_n}$ to satisfy conditions in Definition \ref{def: traversal order} and
\begin{itemize}
\item if $a_n \in \{1,2\}$,
let $f_{a_1,\dots, a_n}=f_{a_1,\dots, a_{n-1}}$.
\item
if $a_n=0$, let $f_{a_1,\dots, a_n}(0) = f_{a_1,\dots,
a_{n-1}}(0)$,
$f_{a_1,\dots, a_n}(1) = f_{a_1,\dots,
a_{n-1}}(3)$,
\item
if $a_n=3$, let $f_{a_1,\dots, a_n}(0) = f_{a_1,\dots,
a_{n-1}}(2)$,
$f_{a_1,\dots, a_n}(1) = f_{a_1,\dots, a_{n-1}}(1).$
\end{itemize}
\end{alg}

Now it
is not hard to give a more precise definition of the standard
Hilbert curve. Let $0.a_1\dots a_n \dots$ denote the $4$-ary
expansion of $t\in[0,1]$. We construct functions $H_n:[0,1]\to
[0,1]^2$ by setting
$\displaystyle{H_1(t)=\left(\frac{1}{2},\frac{1}{2}\right)+\frac{1}{4}f_\emptyset(a_1)}$
and for $n\ge 2$,
\begin{equation}\label{eq: standard Hilbert recursion}
H_{n}(t)=H_{n-1}(t)+\left(\frac{1}{2}\right)^{n+1}
f_{a_1,\dots,a_{n-1}}(a_{n}).
\end{equation}
Let $\displaystyle{H(t)=\lim_{n\to \infty}H_n(t)}$ be the
resulting Hilbert curve. $H$ is well defined, space filling and
 $1/2$-H\"{o}lder continuous (see Chapter 2 in \cite{Sagan}).
Hence, by Theorem \ref{thm:LG}, $H$ is polar for planar Brownian motion.
Also note that the Hilbert curve is self-similar,
through the mappings from the construction, and measure
preserving with Lebesgue measure on both the unit interval and
the unit square.

\subsection{Generalized Hilbert curves}
We now modify this construction in order to obtain curves that
are space filling and H\"{o}lder continuous with the exponent
slightly less than $ 1/2$, but which remain space filling after
perturbation by planar Brownian motion. The curve will be
constructed on overlapping squares, using the same algorithm as
above, but a different parametrization.

\subsubsection{Main idea} Fix $\alpha>1/2$. At level 1, we subdivide the unit
square into four sub-squares of side length $1/2$. Cover each of
these sub-squares with a slightly larger (centered) square of
side length $\alpha$, as in Figure \ref{fig:Hilbert2}. Iterate
the procedure.
Thus, in step $n$, each of the $4^n$ squares of
side length $\alpha^n$ are covered by four overlapping squares of side length
$\alpha^{n+1}$. We want to find a labeling
that determines the order in which these overlapping squares are
traversed, so that we obtain a continuous
function in the limit. Naively, let us first try the labeling
and parametrization used for the standard Hilbert curve, and to
time $k/4^n$ assign one of the $4^n$ squares of side length
$\alpha^n$ as in the Hilbert algorithm.
The curve in the first
iteration is the same as the standard Hilbert curve and the next two iterations are shown in Figure \ref{fig:Hilbert2}.

\begin{figure}[h]
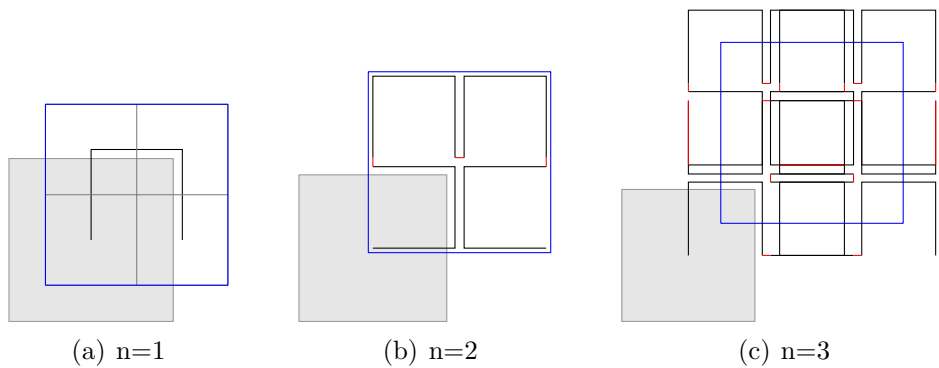

  \begin{center}
    \subfigure[n=1]{
    \tikz[scale=3]{
    \filldraw[color=gray!80, fill=gray!20]
    (-1.6mm,-1.6mm) rectangle (5.6mm,5.6mm);
    \draw[black](2mm,2mm)--(2mm,6mm)--(6mm,6mm)--(6mm,2mm);
    \draw[step=4mm, gray, thin] (0mm,0mm) grid (8mm,8mm);
    \draw[step=8mm, blue, thin] (0mm,0mm) grid (8mm,8mm);
    }
    }
    \hspace{.2in}
    \subfigure[n=2]{\tikz[scale=3]{
    \filldraw[color=gray!80, fill=gray!20](-3.04mm,-3.04mm) rectangle (3.44mm,3.44mm);
    \draw[step=8mm, blue, thin] (0mm,0mm) grid (8mm,8mm);
    \draw[black](.2mm,.2mm)--(3.8mm,.2mm)--(3.8mm,3.8mm)--(.2mm,3.8mm);
   \draw[black](.2mm,4.2mm)--(.2mm,7.8mm)--(3.8mm,7.8mm)--(3.8mm,4.2mm);
   \draw[black](4.2mm,4.2mm)--(4.2mm,7.8mm)--(7.8mm,7.8mm)--(7.8mm,4.2mm);
   \draw[black](7.8mm,3.8mm)--(4.2mm,3.8mm)--(4.2mm,.2mm)--(7.8mm,.2mm);
   \draw[red](.2mm,3.8mm)--(.2mm,4.2mm);
   \draw[red](3.8mm,4.2mm)-- (4.2mm,4.2mm);
   \draw[red](7.8mm,4.2mm)-- (7.8mm,3.8mm);
    } }
    \hspace{.2in}
    \subfigure[n=3]{\tikz[scale=3]{
    \filldraw[color=gray!80, fill=gray!20](-4.336mm,-4.336mm) rectangle (1.496mm,1.496mm);
    \draw[step=8mm, blue, thin] (0mm,0mm) grid (8mm,8mm);
   \draw[black](-1.42mm, -1.42mm)--(-1.42mm, 1.82mm)--(1.82mm, 1.82mm)--(1.82mm, -1.42mm);
   \draw[black](2.18mm,-1.42mm)--(5.42mm,-1.42mm)--(5.42mm,1.82mm)--(2.18mm,1.82mm);
   \draw[black](2.18mm,2.18mm)--(5.42mm,2.18mm)--(5.42mm,5.42mm)--(2.18mm,5.42mm);
   \draw[black](1.82mm,5.42mm)--(1.82mm,2.18mm)--(-1.42mm,2.18mm)--(-1.42mm,5.42mm);
   \draw[black](-1.42mm,2.58mm)--(1.82mm,2.58mm)--(1.82mm,5.82mm)--(-1.42mm,5.82mm);
  \draw[black](-1.42mm,6.18mm)--(-1.42mm,9.42mm)--(1.82mm,9.42mm)--(1.82mm,6.18mm);
  \draw[black](2.18mm,6.18mm)--(2.18mm,9.42mm)--(5.42mm,9.42mm)--(5.42mm,6.18mm);
  \draw[black](5.42mm,5.82mm)--(2.18mm,5.82mm)--(2.18mm,2.58mm)--(5.42mm,2.58mm);
  \draw[black](2.58mm,2.58mm)--(5.82mm,2.58mm)--(5.82mm,5.82mm)--(2.58mm,5.82mm);
  \draw[black](2.58mm,6.18mm)--(2.58mm,9.42mm)--(5.82mm,9.42mm)--(5.82mm,6.18mm);
  \draw[black](6.18mm,6.18mm)--(6.18mm,9.42mm)--(9.42mm,9.42mm)--(9.42mm,6.18mm);
  \draw[black](9.42mm,5.82mm)--(6.18mm,5.82mm)--(6.18mm,2.58mm)--(9.42mm,2.58mm);
  \draw[black](9.42mm,5.42mm)--(9.42mm,2.18mm)--(6.18mm,2.18mm)--(6.18mm,5.42mm);
  \draw[black](5.82mm,5.42mm)--(2.58mm,5.42mm)--(2.58mm,2.18mm)--(5.82mm,2.18mm);
  \draw[black](5.82mm,1.82mm)--(2.58mm,1.82mm)--(2.58mm,-1.42mm)--(5.82mm,-1.42mm);
  \draw[black](6.18mm,-1.42mm)--(6.18mm,1.82mm)--(9.42mm,1.82mm)--(9.42mm,-1.42mm);
\draw[red](-1.42mm,5.42mm)--(-1.42mm,2.58mm);
  \draw[red](5.42mm,2.58mm)--(2.58mm,2.58mm);
  \draw[red](9.42mm,2.58mm)--(9.42mm,5.42mm);
  \draw[red](1.82mm,-1.42mm)--(2.18mm,-1.42mm);
  \draw[red](2.18mm,1.82mm)--(2.18mm,2.18mm);
  \draw[red](2.18mm,5.42mm)--(1.82mm,5.42mm);
  \draw[red](-1.42mm,5.82mm)--(-1.42mm,6.18mm);
  \draw[red](1.82mm,6.18mm)--(2.18mm,6.18mm);
  \draw[red](5.42mm,6.18mm)--(5.42mm,5.82mm);
  \draw[red](2.58mm,5.82mm)--(2.58mm,6.18mm);
  \draw[red](5.82mm,6.18mm)--(6.18mm,6.18mm);
  \draw[red](9.42mm,6.18mm)--(9.42mm,5.82mm);
  \draw[red](6.18mm,5.42mm)--(5.82mm,5.42mm);
  \draw[red](5.82mm,2.18mm)--(5.82mm,1.82mm);
  \draw[red](5.82mm,-1.42mm)--(6.18mm,-1.42mm);
  } }
\end{center}
  \caption{First 3 steps in the new construction with $\alpha=.9$. The blue square is the
unit square, red lines are connections between copies of lower
level iterations, and the gray squares denote typical covering
squares at each level.}
  \label{fig:Hilbert2}
\end{figure}

The limiting function is space filling, but it is \emph{not
continuous}. To see this, first note that following the same
algorithm as in the Hilbert curve construction leads to
eventually starting and ending the curve outside of the unit
square. In figure \ref{fig:Hilbert2}, when $\alpha=.9$, this
happens already in the third step.
For a given $\alpha$, there will be an $\epsilon$ and an integer
$m$ such that the starting and the ending point of the curve
after $m^{th}$ iteration is at least $\epsilon$ distance, in
each coordinate direction, outside the unit square.
Following the standard Hilbert algorithm the endpoints of the part of the curve in the top-left sub-square should converge to the image of $1/2$. The same is true for the starting points of the part of curve in the top-right sub-square. However after $m^{th}$ iteration these starting points and endpoints will be (in each coordinate direction) at distance at least $\alpha \epsilon$ from the center of the unit square. By symmetry their distance will be at least $2\alpha \epsilon$.
Therefore their limits will not be equal which implies the discontinuity at $1/2$.

The main idea in our construction is the following: in each
iteration, we assign three intervals of time reserved for
connecting the four descendents. Choosing the right way to
assign these connection times, along with the appropriate
expansion factor for the squares, will allow us to obtain a
H\"{o}lder continuous curve of any exponent less than $1/2$. The
details for this construction can be found below.

\subsubsection{Scaled squares}\label{sec:boxes}
 Let $0<\alpha <1$ be
given and set $\beta = \frac{2\alpha - 1}{4}$ which might be negative if $\alpha < 1/2$.

Analogous to the standard Hilbert curve construction, the set of
squares at iteration $n+1$, denoted by $\mathcal{F}_{n+1}$, is
constructed from $\mathcal{F}_n$ by taking
$\mathcal{F}_0=\left\{[0,1]^2\right\}$ and if $[a,a+\delta]\times [b,b+\delta]$
is in $\mathcal{F}_n$, its descendents in $\mathcal{F}_{n+1}$
are
$$
[a+(i/2-\beta)\delta,a+((i+1)/2+ \beta)\delta] \times [b+(j/2-\beta)\delta,
b+((j+1)/2+\beta)\delta],
$$
for $i,j \in \{0,1\}$.
 We define the partial ordering on $\cup_{n\geq 0}\mathcal{F}_n$ by setting
 $Q_1 \leq Q_2$ if $Q_2$ is a descendent of $Q_1$ and extending it by transitivity.
Note that the side length of squares in $\mathcal{F}_n$ is equal
to $\alpha^n$ and if $(x,y)$ is the center of a square in
$\mathcal{F}_n$ then the centers of its four descendents are
given by $(x,y) + \frac{1}{4}\alpha^n (i,j)$, where
$i,j\in\{-1,1\}$. If $Q_1\in \mathcal{F}_n$ then $\cup_{Q_1 \leq
Q_2}Q_2$ is equal to $Q_1$ if $\alpha \leq 1/2$ and  otherwise
is a square with the same center as $Q_1$, but with side length
$(1+\frac{2\beta}{1-\alpha})\alpha^n =
\frac{\alpha^n}{2(1-\alpha)}.$ We denote the square $\cup_{Q_1
\leq Q_2}Q_2$ by $\scale(Q_1)$. In particular if $\alpha > 1/2$,
the union of all squares, over all generations, gives the square
$[-\frac{\beta}{1-\alpha},1+\frac{\beta}{1-\alpha}]^2.$

\subsubsection{Generalized Hilbert curve - construction}\label{sec:type I}
Let $\rho<1/4$ and $\rho < \alpha < 1$ be given. First we construct a
Cantor set $\textsf{C}_{\rho}$ as follows. At stage 1, divide
the unit length interval into 7 subintervals out of which we
keep four closed intervals of length $\rho$. The collection of
the intervals still kept at stage 1 is
\[\textsf{C}_{\rho}^1:=\left\{
[0,\rho], \left[\frac{1-\rho}{3},\frac{1+2\rho}{3}\right],
\left[\frac{2-2\rho}{3},\frac{2+\rho}{3}\right],
[1-\rho,1]\right\}.\] The removed open intervals will be called
\emph{connection intervals}. At stage $n$, each remaining closed
interval of length $\rho^{n-1}$ from $\textsf{C}_{\rho}^{n-1}$
is divided into seven subintervals out of which four closed
intervals are being kept. For example, if $[a,a+\rho^{n-1}]\in
\textsf{C}_{\rho}^{n-1}$, then we have four descendents of this
interval
\[\left\{
\left[a+\frac{k}{3}(1-\rho)\rho^{n-1}\, , \,
a+\rho^n+\frac{k}{3}(1-\rho)\rho^{n-1}\right]: k=0,1,2,3
\right\} \subset \textsf{C}_{\rho}^n.\]

Let $\textsf{C}_{\rho}=\cap_{n=1}^{\infty} \textsf{C}_{\rho}^n$
be the resulting Cantor set. For every $t\in \textsf{C}_{\rho}$,
there exists a unique sequence $\{a_n\}_{n=1}^{\infty}$ with $a_n\in
\{0,1,2,3\}$, such that
$t=\frac{1-\rho}{3\rho}\sum_{n=1}^{\infty}a_n \rho^n$. Let us
call this sequence the \emph{$\rho$-expansion} of $t$.

Recursively define the sequence of functions $\tg_n \colon
\textsf{C}_{\rho} \to \mathbb{R}^2$ as follows. Abusing
notation, denote the $\rho$-expansion of $t$ by $0.a_1a_2\dots$
and let $\tg_1 (t) = (\frac{1}{2},\frac{1}{2}) +
\frac{1}{4}f_\emptyset(a_1)$. For $n \geq 2$, let
\begin{equation}\label{eq: generalized Hilbert recursion}
\tg_n(t)=\tg_{n-1}(t)+\frac{\alpha^{n-1}}{4}f_{a_1,\dots,a_{n-1}}(a_n),
\end{equation}
where the functions $f_{a_1,\dots,a_n}$ are defined according to
Algorithm \ref{alg:main}.
Let $\displaystyle{\tg(t)=\lim_{n\to\infty}\tg_n(t)}$ for all
$t\in \textsf{C}_{\rho}$ and linearly interpolate between these
values to obtain a continuous curve on $[0,1]$. It is an easy
exercise to find a formula for $\tg(t)$ when $t\notin
\textsf{C}_{\rho}$, but we will not need it here. The function
is clearly well defined and continuous. We will show next that
$G$ is H\"{o}lder continuous and moreover space-filling (its range covers an open set) for $\alpha \geq 1/2$.

\subsubsection{Properties of generalized Hilbert curves}\label{Properties of generalized Hilbert curves}
The range of $\tg_n$ will be exactly the set of centers of the
squares in $\mathcal{F}_n$.
For any $t \in \textsf{C}_{\rho}$ whose $\rho$-expansion starts
with $0.a_1 \dots a_n$ denote by $D_n(t)$ the set of numbers in
$\textsf{C}_{\rho}$ whose $\rho$-expansion starts with the same
$n$ digits. If $G_n(t)$ is the center of $Q \in \mathcal{F}_n$
then numbers $r \in D_n(t)$ will be mapped by $G_{n+1}$ into the
centers of descendents of $Q$. Therefore $G_{n+k}(r) \in \scale
(Q)$ for all $k \geq 0$ and $r \in D_n(t)$, in particular after
taking limits $G(r) \in \scale (Q)$.

We claim $\tg$ is space filling if $\alpha \geq 1/2$, more precisely, the image of
$G$ covers the entire square $\scale([0,1]^2) =
[-\frac{\beta}{1-\alpha},1+\frac{\beta}{1-\alpha}]^2$. Let $U$ be
an open set that intersects the interior of the square
$[-\frac{\beta}{1-\alpha},1+\frac{\beta}{1-\alpha}]^2$. Then $U$
contains a square $\scale(Q)$ for some $Q \in
\mathcal{F}_n$, and for a sufficiently large $n$. As argued
before, the image of $\tg$ will intersect $\scale(Q)$ and
thus $U$. Therefore, the image of $\tg$ is dense in
$[-\frac{\beta}{1-\alpha},1+\frac{\beta}{1-\alpha}]^2$ and by
compactness, it is equal to the whole square
$[-\frac{\beta}{1-\alpha},1+\frac{\beta}{1-\alpha}]^2$.

We will now compute the H\"{o}lder exponent for $\tg$. First it
is important to note that if the interval
$(a,a+\frac{1-4\rho}{3}\rho^n)$ is a connection interval removed
at level $(n+1)$ in the construction of $\textsf{C}_{\rho}$ then
$|\tg(a)-\tg(a+\frac{1-4\rho}{3}\rho^n)|=C\alpha^n$, where $C = \frac{2\beta}{1-\alpha}$ if $\alpha>1/2$.
Also note that 
the
endpoints of connection intervals belong to $\textsf{C}_{\rho}$.

Let $t, s$ in $[0,1]$, with $|t-s| < \frac{1-4\rho}{3}\rho^{n-1}$
be given. There are a few cases to be considered.
\begin{itemize}
\item[(i)] If $t, s\in
\textsf{C}_{\rho}$, then they have common first $n$ digits in their $\rho$-expansions.
Hence $\tg_n(t) = \tg_n(s)$ 
and by the
discussion above, there exists a constant $C_1$, depending only
on $\alpha,$ such that
\[|\tg(t)-\tg(s)|\le C_1 \,\alpha^n.\]
\item[(ii)] If $t, s \notin \textsf{C}_{\rho}$
and they belong to the same connection interval
$(a,a+\frac{1-4\rho}{3}\rho^k)$ for some $k$, then
\[|\tg(t)-\tg(s)|=
\frac{|\tg(a)-\tg(a+\frac{1-4\rho}{3}\rho^k)|}{\frac{1-4\rho}{3}\rho^k}
|t-s| \le C_2\, \alpha^n,\] for some constant $C_2$ which
depends on $\alpha$ and $\rho$. Here we have used the fact that
$\alpha/\rho>1$.
The same argument holds for $s$ inside a
connection interval and $t$ being one of the endpoints of this
interval.
\item[(iii)] If $t\in
\textsf{C}_{\rho}$ but $s\notin \textsf{C}_{\rho}$, and $t<s$,
let $a$ be the left endpoint of the connection interval
containing $s$. Then
\[|\tg(t)-\tg(s)|\le |\tg(t)-\tg(a)|+|\tg(a)-\tg(s)|\le C_1\,
\alpha^n+C_2\,\alpha^n,\] from cases (i) and (ii). Similarly, we
use the right endpoint of the connection interval if $t>s$.
\item[(iv)] If $t, s \notin \textsf{C}_{\rho}$ but they belong
to disjoint connection intervals, and $t<s$, we let $a$ be the
right endpoint of the connection interval containing $t$. Then
using cases (ii) and (iii),
\[|\tg(t)-\tg(s)|\le |\tg(t)-\tg(a)|+|\tg(a)-\tg(s)|\le C_1\,
\alpha^n+2C_2\,\alpha^n.\]
\end{itemize}
From here we see that $\tg$ is $\frac{\log\alpha}{\log\rho}$- H\"{o}lder continuous. Note that $\rho$ can be made
arbitrarily close to $1/4$ and if we let $\alpha\downarrow 1/2$
and $\rho\uparrow 1/4$ simultaneously, we can get a  H\"{o}lder
exponent arbitrarily close to $1/2$.

\begin{remark}\label{remark: sharp continuity}
Choosing different expansion scales for different iterations,
one can construct a single function that is $\gamma$-H\"{o}lder for all
$\gamma<1/2$. In the $n$-th level of the construction we need to take $\alpha_n$ as the expansion factor for the squares and $\rho_n$ as scaling factor for the Cantor set so that $\alpha_n \downarrow 1/2$ and $\rho_n \uparrow 1/4$.
\end{remark}

\section{Proof of Theorem \ref{thm:main} ($d=2$)}

Let $B_t$ be a planar Brownian motion. Almost surely, there is a
(random) $\epsilon$ such that, for any $s \leq \epsilon$ and any
$0 \leq t \leq 1-s$, we have $|B_{t+s} - B_t| \leq
3\sqrt{s\log(1/s)}$ (see Theorem 1.14 in \cite{PM}). The result
now follows from the theorem below.

\begin{theorem}\label{thm: main general version}
Let $\tg$ be the generalized Hilbert curve constructed with
$\alpha>1/2$. Let $h \colon [0,1] \to \mathbb{R}^2$ be a
function for which there is a constant $C$ such that, for $s$
small enough and $0 \leq t \leq 1-s$, we have $|h(t+s) - h(t)|
\leq C\sqrt{s\log (1/s)}$. Then the range of the function
$\tg+h$ covers an open set.
\end{theorem}
\begin{proof}
Pick $n_0$ large enough such that the assumption on $h$ holds
for all $s \leq \rho^{n_0-1}$ and such that $\beta \alpha^{n}
\geq C\sqrt{\rho^{n}n\log(1/\rho)}$ for all $n \geq n_0$.
Let $Q_n(p)$ denote a square of side
length $\alpha^n$ centered at $p$, and recall the definition of $D_n(t)$ from
subsection \ref{Properties of generalized Hilbert curves}.
By construction, for any $t \in \textsf{C}_{\rho}$, the square
$Q_n(\tg_n(t)+h(t))$ is covered by the squares
$Q_{n+1}(\tg_{n+1}(r)+h(t))$, for $r \in D_n(t)$, even when each
of them is shifted by no more than $\beta \alpha^{n+1}$. Now by
assumption on $h$, the squares $Q_{n+1}(\tg_{n+1}(r)+h(r))$, $r
\in D_n(x)$ cover the square $Q_{n}(\tg_n(t)+h(t))$ whenever $n
\geq n_0$. Therefore, for all $m \geq  n \geq n_0$  the square
$Q_{n}(\tg_{n}(t)+h(t))$ is covered by the squares
$Q_m(\tg_m(r)+h(r)),$ where $r$ ranges over the set $D_n(t)$.

Let $U$ be an open set that intersects
$Q_{n_0}(\tg_{n_0}(t)+h(t))$ for some $t \in \textsf{C}_{\rho}$.
By the discussion above, one can find $m$ large enough and $r
\in D_{n_0}(t)$ such that the square $Q$ with the center at
$\tg_m(r)+h(r)$ and side length equal to $C\sqrt{\rho^{m}m\log
(1/\rho)}+ \frac{\alpha^m}{2(1-\alpha)}$ lies in $U$. If $s \in
D_m(r)$, then $\tg(s)+h(s)$ lies in $Q$ and thus in $U$. This
implies that the image of $\tg+h$ is dense in the square
$Q_{n_0}(\tg_{n_0}(t)+h(t))$ and, by compactness, it covers it.
\end{proof}

\begin{remark}\label{remark: sharp estimates}
As mentioned in Remark \ref{remark: sharp continuity}, by taking different $\alpha_n$ and $\rho_n$ at different stages in the construction of the Hilbert curve, we can obtain a generalized Hilbert curve which is $\gamma$-H\"{o}lder continuous for all $\gamma < 1/2$. Actually such curves can be constructed in a way so that they satisfy the two-dimensional case of Theorem \ref{thm:main}. To do this we only need to ensure that Brownian motion does not destroy the overlap of squares at different levels and the condition for this is that for large enough $n$ we have
$$
\beta_{n+1}\prod_{i=1}^n \alpha_i \geq C\sqrt{-\Big(\prod_{i=1}^{n}\rho_i\Big)\log \Big(\prod_{i=1}^{n}\rho_i\Big)},
$$
where $\beta_{n+1} = \frac{2\alpha_{n+1}-1}{4}$.
Such sequences $\alpha_n$ and $\rho_n$ can be constructed with $\alpha_n \downarrow 1/2$ and $\rho_n \uparrow 1/4$. For example take $\alpha_n = \frac{e^{1/n}}{2}$ and $\rho_n = \frac{e^{-1/n}}{4}$.
\end{remark}

\section{Construction of a reverse H\"{o}lder continuous drift}
In this section we describe an alternate construction of a
function which added to standard Brownian motion will give an
almost surely space-filling curve. While its construction is
slightly more complicated, this function has the advantage that
it is reverse H\"{o}lder continuous.
\begin{definition}\label{def; lower holder}
We say that a continuous function $f \colon \mathbb{R} \to \mathbb{R}^d$ is \textit{reverse $\alpha$-H\"{o}lder continuous}, for $0 < \alpha < 1$, if for some $C>0$ and any open interval $I$ of length $|I|$ we can find $s,t \in I$ such that $|f(s)-f(t)| \geq C|I|^\alpha$.
\end{definition}
The notion of reverse H\"{o}lder continuity function is closely
related to the geometric properties of its graph. For example
graphs of functions which are both  $\alpha$-H\"{o}lder and
reverse $\alpha$-H\"{o}lder continuous (like the function we
will construct in this section) have Hausdorff dimension bigger
than 1, see Theorem 4 in \cite{PU}. Clearly, the previously
constructed generalized Hilbert curve is not reverse H\"{o}lder
continuous since it is linear outside the Cantor set
$\textsf{C}_{\rho}$.

Let $\alpha>1/2$ be given and recall that $\beta=(2\alpha-1)/4$.
Let $S$ be a set of positive integers such that $\sum_{i \in
S}\alpha^i = \frac{2\beta}{1-\alpha}$. This is the set of
\emph{exceptional times} when our construction deviates from the
standard Hilbert algorithm. One way to construct this subset is
to define $\lambda_1 :=\frac{2\beta}{1-\alpha}$ and, at step
$i$, include in $S$ the smallest integer $j$, larger than all
integers already in $S$, which has the property $\alpha^j \leq
\lambda_i$ and set $\lambda_{i+1}=\lambda_i - \alpha^j$. We call
the string $a_1, \dots, a_n$ $0$-\textit{reverse}
($3$-\textit{reverse}) if $a_1, \dots, a_{n}$ ends in exactly
$k$ $0$'s ($k$ $3$'s) for some $k$ with $k+1 \in S$. This means
$a_{n-k+1} = \dots = a_{n}= 0 \neq a_{n-k}$, and similarly for
$3$'s. For a string that ends in several consecutive $0$'s
($3$'s) we define its \textit{prestring} to be the string
obtained by deleting the maximal block of consecutive $0$'s
($3$'s) at the end.

\begin{alg}[Alternate generalized Hilbert]
Define $f_\emptyset$ by (\ref{eq: starting function}).
Let $a_1,\dots, a_n$ be given and let $a_1, \dots, a_m$ be its
prestring. Given $f_{a_1,\dots, a_{n-1}}$ and $a_n$ define $f_{a_1, \dots , a_n}$ to satisfy the conditions in Definition \ref{def: traversal order} and
\begin{itemize}
\item if $a_n \in \{1,2\}$, let $f_{a_1,\dots a_n}=f_{a_1,\dots, a_{n-1}}$.
\item if $a_n=0$, let $f_{a_1,\dots, a_n}(1) = f_{a_1,\dots,
a_{n-1}}(3)$ and
\begin{itemize}
\item if $a_1,\dots, a_n$ is $0$-reverse, let $f_{a_1,\dots, a_n}(0) =
f_{a_1,\dots, a_m}(2)$.
\item if $a_1,\dots, a_n$ is not $0$-reverse,
let $f_{a_1,\dots, a_n}(0) = f_{a_1,\dots, a_m}(0)$.
\end{itemize}
\item if $a_n=3$, let $f_{a_1,\dots, a_n}(0) =
f_{a_1,\dots, a_{n-1}}(2)$ and
\begin{itemize}
\item if $a_1,\dots, a_n$ is $3$-reverse, let $f_{a_1,\dots, a_n}(1) =
f_{a_1,\dots, a_m}(3)$.
\item if $a_1,\dots, a_n$ is not $3$-reverse, let $f_{a_1,\dots, a_n}(1) =
f_{a_1,\dots, a_m}(1)$.
\end{itemize}
\end{itemize}
\end{alg}

For $n \geq 1$ let $\mathbb{Q}_n=\{k4^{-n}: k \in
\mathbb{Z}, 0 \leq k\leq 4^{n} \}$.
For every $0 \leq t \leq 1$ consider its $4$-ary
expansion $t=0.a_1 a_2 \dots$. Now define $\widetilde{G}_0$ to
be the constant function mapping into the center of the unit
square $[0,1]^2$ and for $n\ge 1$, let the sequence of functions
$\widetilde{G}_n$ be given by recursion (\ref{eq: generalized
Hilbert recursion}). This is a Cauchy sequence of functions and
so finally define $\widetilde{G}$ as their limit.


\begin{figure}[h]
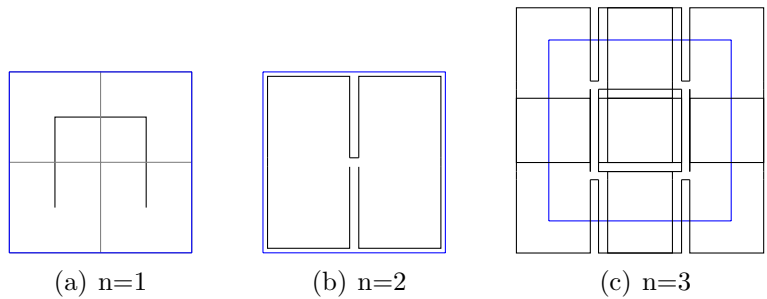

  \begin{center}
    \subfigure[n=1]{
    \tikz[scale=3]{
    \draw[black](2mm,2mm)--(2mm,6mm)--(6mm,6mm)--(6mm,2mm);
    \draw[step=4mm, gray, thin] (0mm,0mm) grid (8mm,8mm);
    \draw[step=8mm, blue, thin] (0mm,0mm) grid (8mm,8mm);
    }
    }
    \hspace{.2in}
    \subfigure[n=2]{\tikz[scale=3]{
    \draw[step=8mm, blue, thin] (0mm,0mm) grid (8mm,8mm);
    \draw[black](3.8mm,3.8mm)--(3.8mm,.2mm)--(.2mm,.2mm)--(.2mm,3.8mm);
   \draw[black](.2mm,4.2mm)--(.2mm,7.8mm)--(3.8mm,7.8mm)--(3.8mm,4.2mm);
   \draw[black](4.2mm,4.2mm)--(4.2mm,7.8mm)--(7.8mm,7.8mm)--(7.8mm,4.2mm);
   \draw[black](7.8mm,3.8mm)--(7.8mm,.2mm)--(4.2mm,.2mm)--(4.2mm,3.8mm);
   \draw[black](.2mm,3.8mm)--(.2mm,4.2mm);
   \draw[black](3.8mm,4.2mm)-- (4.2mm,4.2mm);
   \draw[black](7.8mm,4.2mm)-- (7.8mm,3.8mm);
    } }
    \hspace{.2in}
    \subfigure[n=3]{\tikz[scale=3]{
    \draw[step=8mm, blue, thin] (0mm,0mm) grid (8mm,8mm);
   \draw[black](5.42mm,5.42mm)--(2.18mm,5.42mm)--(2.18mm,2.18mm)--(5.42mm,2.18mm);
   \draw[black](5.42mm,1.82mm)--(5.42mm,-1.42mm)--(2.18mm,-1.42mm)--(2.18mm,1.82mm);
   \draw[black](1.82mm, 1.82mm)--(1.82mm, -1.42mm)--(-1.42mm, -1.42mm)--(-1.42mm, 1.82mm);
   \draw[black](-1.42mm,2.18mm)--(-1.42mm,5.42mm)--(1.82mm,5.42mm)--(1.82mm,2.18mm);
   \draw[black](5.42mm,2.18mm)--(5.42mm,1.82mm);
   \draw[black](2.18mm,1.82mm)--(1.82mm, 1.82mm);
   \draw[black](-1.42mm, 1.82mm)--(-1.42mm,2.18mm);
   \draw[black](1.82mm,5.82mm)--(1.82mm,2.58mm)--(-1.42mm,2.58mm)--(-1.42mm,5.82mm);
  \draw[black](-1.42mm,6.18mm)--(-1.42mm,9.42mm)--(1.82mm,9.42mm)--(1.82mm,6.18mm);
  \draw[black](2.18mm,6.18mm)--(2.18mm,9.42mm)--(5.42mm,9.42mm)--(5.42mm,6.18mm);
  \draw[black](5.42mm,5.82mm)--(5.42mm,2.58mm)--(2.18mm,2.58mm)--(2.18mm,5.82mm);
  \draw[black](-1.42mm,5.82mm)--(-1.42mm,6.18mm);
  \draw[black](1.82mm,6.18mm)--(2.18mm,6.18mm);
  \draw[black](5.42mm,6.18mm)--(5.42mm,5.82mm);
  \draw[black](5.82mm,5.82mm)--(5.82mm,2.58mm)--(2.58mm,2.58mm)--(2.58mm,5.82mm);
  \draw[black](2.58mm,6.18mm)--(2.58mm,9.42mm)--(5.82mm,9.42mm)--(5.82mm,6.18mm);
  \draw[black](6.18mm,6.18mm)--(6.18mm,9.42mm)--(9.42mm,9.42mm)--(9.42mm,6.18mm);
  \draw[black](9.42mm,5.82mm)--(9.42mm,2.58mm)--(6.18mm,2.58mm)--(6.18mm,5.82mm);
  \draw[black](2.58mm,5.82mm)--(2.58mm,6.18mm);
  \draw[black](5.82mm,6.18mm)--(6.18mm,6.18mm);
  \draw[black](9.42mm,6.18mm)--(9.42mm,5.82mm);
  \draw[black](6.18mm,2.18mm)--(6.18mm,5.42mm)--(9.42mm,5.42mm)--(9.42mm,2.18mm);
  \draw[black](9.42mm,1.82mm)--(9.42mm,-1.42mm)--(6.18mm,-1.42mm)--(6.18mm,1.82mm);
  \draw[black](5.82mm,1.82mm)--(5.82mm,-1.42mm)--(2.58mm,-1.42mm)--(2.58mm,1.82mm);
  \draw[black](2.58mm,2.18mm)--(5.82mm,2.18mm)--(5.82mm,5.42mm)--(2.58mm,5.42mm)
;
  \draw[black](9.42mm,2.18mm)--(9.42mm,1.82mm);
  \draw[black](6.18mm,1.82mm)--(5.82mm,1.82mm);
  \draw[black](2.58mm,1.82mm)--(2.58mm,2.18mm);
\draw[black](1.82mm,2.18mm)--(1.82mm,5.82mm);
\draw[black](2.18mm,5.82mm)--(5.82mm,5.82mm);
\draw[black](6.18mm,5.82mm)--(6.18mm,2.18mm);
  } }
\end{center}
  \caption{First 3 steps in the new construction with $\alpha=.9$. Note that the curves start and
end inside the unit square, drawn in blue.}
  \label{fig:Hilbert3}
\end{figure}

Space filling property of this function is obvious.

Now we prove that the constructed function is both $\frac{\log
(1/\alpha)}{\log 4}$-H\"{o}lder and reverse $\frac{\log
(1/\alpha)}{\log 4}$-H\"{o}lder continuous. For the reverse
H\"{o}lder continuity notice that any interval $I\subset [0,1]$
contains a $4$-ary interval $J$, of the form $[k4^{-n},
(k+1)4^{-n}]$, of length at least $|I|/8$. Moreover any
$4$-ary interval $J$ of length $4^{-n}$ is mapped by
$\widetilde{G}$ onto a square of side length
$\frac{\alpha^n}{2(1-\alpha)} = \frac{1}{2(1-\alpha)}|J|^{\log
(1/\alpha)/\log 4}$. Pick $s,t \in J$ so that $\widetilde{G}(s)$
and $\widetilde{G}(t)$ are vertices of one side of the square
$\widetilde{G}(J)$. We see that
$$
|\widetilde{G}(s)-\widetilde{G}(t)| = \frac{1}{2(1-\alpha)}|J|^{\log (1/\alpha)/\log 4} \geq \frac{\alpha^{3/2}}{2(1-\alpha)}|I|^{\log(1/\alpha)/\log 4}.
$$
If $s \in
\mathbb{Q}_n$ and $t$ is any number whose $4$-ary expansion
agrees in the first $n$ digits with $s$, then
$|\widetilde{G}(t)-\widetilde{G}(s)| \leq
\frac{\alpha^n}{2(1-\alpha)}$. Take any $0 \leq t_1 < t_2 \leq
1$, and $n$ such that $t_2-t_1 \leq 4^{-n}$ and $q_1, q_2 \in
\mathbb{Q}_n$ such that $0 \leq t_i - q_i \leq 4^{-n}$, for
$i=1,2$. Then we have $|\widetilde{G}(t_i) - \widetilde{G}(q_i)|
\leq \frac{\alpha^n}{2(1-\alpha)}$ for $i=1,2$ and
$|\widetilde{G}(t_2)-\widetilde{G}(t_1)| \leq
\frac{\alpha^n}{1-\alpha} +
|\widetilde{G}(q_1)-\widetilde{G}(q_2)|$. Since either $q_1 =
q_2$ or $q_2 = q_1+4^{-n}$ we see that it suffices to prove the
H\"{o}lder continuity for points in $\bigcup_{n}\mathbb{Q}_n$.

Let $s,t \in \mathbb{Q}_n$ with $t=s+4^{-n}$. If the $4$-ary
expansion of $s$ does not end in $3$ then $s$ and $t$ have equal
first $n-1$ digits and so $|\widetilde{G}(t)-\widetilde{G}(s)|
\leq \frac{1}{\alpha(1-\alpha)} \alpha^n$ and we are done.
Otherwise let the $4$-ary expansion of $s$ be given by $s=0.a_1
\dots a_m 3 \dots 3$ with $k=n-m$ threes in the end and $a_m
\neq 3$. Then obviously $t$ can be written in base $4$ as
$t=0.a_1 \dots a_{m-1}(a_m+1)0 \dots 0$. Define $s_l = 0.a_1
\dots a_m 3 \dots 3$ with $l$ threes in the end, so in
particular $s_k = s$. Without loss of generality assume $f_{a_1,
\dots , a_{m-1}} = f_\emptyset$. Denote the point
$\widetilde{G}_{m-1}(t) = \widetilde{G}_{m-1}(s)$ by $p$. Let us
verify the claim for $a_m=1$ since cases $a_m \in \{0,2\}$ are
similar.

If $a_m=1$, then $f_{a_1, \dots, a_{m}} = f_\emptyset = f_{a_1, \dots,
(a_{m}+1)}$ and, relative to $p$, the coordinates
of $\widetilde{G}_m(s)$ and $\widetilde{G}_m(t)$ are
$(-\frac{\alpha^{m-1}}{4},\frac{\alpha^{m-1}}{4})$ and
$(\frac{\alpha^{m-1}}{4},\frac{\alpha^{m-1}}{4}),$ respectively.
Now, for any $l>0$, the point $\widetilde{G}_{m+l}(s_{l})$ is constructed
from $\widetilde{G}_{m+l-1}(s_{l-1})$ by shifting it by
$\frac{\sqrt{2}}{4}\alpha^{m+l-1}$ in the lower right direction,
unless $l \in S$ when we have to shift the point by
the same distance in the upper left direction.
From the definition of $S$ it is clear that $\widetilde{G}_{m+l}(s_{l})$
will converge to $p$ as $l \to \infty$ and moreover that $|\widetilde{G}_{m+l}(s_l) - p| \leq \sqrt{2}\alpha^m\sum_{j \geq l}\alpha^j = \frac{\sqrt{2}}{1-\alpha}\alpha^{m+l}$. In particular for $l=k$ this implies that
\begin{equation}\label{eq: approximation in the construction}
|\widetilde{G}(s) - p| \leq |\widetilde{G}(s) - \widetilde{G}_{n}(s)| + |\widetilde{G}_n(s)-p| \leq \frac{1+2\sqrt{2}}{2(1-\alpha)}\alpha^n = \frac{1+2\sqrt{2}}{2(1-\alpha)} (t-s)^{\log(1/\alpha)/\log 4}.
\end{equation}
Similarly, we argue that $\widetilde{G}_{m+l}(t)$ also converges to $p$ which implies $\widetilde{G}(t)=p$ and plugging this into (\ref{eq: approximation in the construction}) gives the wanted bound.

Adding $\tilde{G}$ to planar Brownian motion results in a space
filling process. The proof of this fact is analogous to the
proof of Theorem \ref{thm:main} when $d=2$, hence we do not
include it here. In $d$-dimensions ($d\ge 3$), one can also
construct a continuous curve with H\"{o}lder exponent
arbitrarily close to $1/d$,
with the same properties as $\tilde{G}$, by changing the
traversal algorithm given in the $d$-dimensional Hilbert curve
construction. Such a higher dimensional curve, analogous to
$\tilde{G}$, can be described using a modification of Butz's
algorithm from \cite{Butz1}.

\section{Proof of Theorem \ref{thm:main} ($d>2$)}
To prove Theorem \ref{thm:main} in dimensions higher than 2, we
construct generalized Hilbert curves in the same way as in two
dimensions.
There are algorithms that generate standard $d$-dimensional
Hilbert curves in dimensions higher than two, which can be
realized through the labeling algorithm we described for two
dimensions. One example is the Butz algorithm, see \cite{Butz2}
and \cite{Butz1}. These algorithms are completely determined by
functions $f_{a_1, \dots , a_n}$ where $a_i \in \{0,1, \dots ,
2^d-1\}$. Throughout this section $f_{a_1, \dots , a_n}$ will
stand for these functions. The Hilbert curve is defined as the
limit of the functions defined recursively by formulae analogous
to (\ref{eq: standard Hilbert recursion}). As in two dimensional
case, the Hilbert curve is again measure preserving. We do not
put here the explicit algorithms for generating functions
$f_{a_1, \dots , a_n}$ as they are relatively complicated. See
also \cite{Sagan} for a geometric description of space filling
curves in higher dimensions.

Let $\rho<2^{-d}$ be given and define $\textsf{C}_{\rho,d}$ by
iteratively removing, at step $n+1$, exactly $2^{nd}(2^d-1)$
(open) connection intervals of length
$\frac{1-(2^d)\rho}{2^d-1}\rho^{n}$ and keeping $2^{(n+1)d}$
intervals of length $\rho^{n+1}$. As in the 2-dimensional case,
every $t\in \textsf{C}_{\rho,d}$ has a unique representation
\[t=K\sum_{n=1}^{\infty}a_n\rho^n,
\hspace{.5in} a_n\in\{0,1\dots,2^d-1\},\] with
$K=\frac{1-\rho}{(2^d-1)\rho}$. We will call $0.a_1a_2\dots$ the
$(\rho,d)-expansion$ of $t$.
Fix an $\alpha>1/2$ and define functions $\tg_0^d$ to be the constant function mapping to the center of the unit cube $[0,1]^d$ and $\tg_n^d$ recursively by (\ref{eq: generalized Hilbert recursion}), with the understanding that the
traversal algorithm and hence the functions $f_{a_1,\dots,a_n}$
correspond to the $d$-dimensional Hilbert curve. Then define $\tg^d$ as the limit of functions $\tg_n^d$.
Again, for
$t\notin \textsf{C}_{\rho,d},$ the values $\tg^d(t)$ are defined
by interpolation. The resulting curve is clearly well-defined
and continuous. 
 H\"{o}lder continuity and space covering property for $\tg^d$ follows by the same argument
as in two dimensional case. 
 The H\"{o}lder exponent is again
$\frac{\log\alpha}{\log\rho}$. We leave the details to the
reader. Taking $\rho$ close to $1/2^d$, and $\alpha$ close to
$1/2$, leads to continuous functions with H\"{o}lder exponents
arbitrarily close to $1/d$.


\begin{proof}[Proof of Theorem \ref{thm:main} ($d>2$)]
As in the two dimensional case, there exists a constant $C>0$
such that, almost surely, for every small enough $s>0$ and all
$0\le t\le 1-s$,
\begin{equation}\label{MC}
|B_{t+s}-B_t|\le C\sqrt{s\log(1/s)}.
\end{equation}
 Suppose $n_0$ is large enough
such that condition \eqref{MC} holds for $s\le \rho^{n_0-1}$ and such
that, for all $n\ge n_0$,
\begin{equation}\label{C2}
\beta\alpha^n\ge C \sqrt{\rho^{n}n\log(1/\rho)}.
\end{equation}
For $t\in \textsf{C}_{\rho,d}$ with $(\rho,d)$-expansion starting with
$0.a_1\dots a_n$ define $D_n(t)$ as the set of numbers $r \in \textsf{C}_{\rho,d}$ whose $(\rho,d)$-expansion starts with the same digits.
Take the box $Q_{n_0}$ of side length
$\alpha^{n_0}$ centered at $B_t-\tg^d_{n_0}(t)$. Just as
in the proof of Theorem \ref{thm: main general version}, it is easy to see
that 
the $2^d$ boxes of
side length $\alpha^{n_0+1}$ centered at $B_r-\tg^d_{n_0+1}(r)$, for $r \in D_{n_0}(t)$,
cover $Q_{n_0}$. Iterating, we see that for all $m\ge 1,$ the box
$Q_{n_0}$ is covered by the $2^{md}$ boxes of side length
$\alpha^{n_0+m}$ centered at
$B_r-\tg^d_{n_0+m}(r),$ for $r \in D_{n_0+m}(t).$ Then
given an arbitrary point $x\in Q_{n_0}$, one can find a sequence of
boxes $\{Q_{n_0+m}\}_{m\ge 1}$, of side length $\alpha^{n_0+m}$, centered as above, that contain
$x$. From here it easily follows that $x$ is in the image of $B-\tg^d$, which completes the
proof.
\end{proof}

\begin{remark}\label{remark; hausdorff}
Let $G$ be a generalized $d$-dimensional Hilbert curve with
$\alpha < 1/2$.
These cases are also interesting since $B-G$ is not space
filling and moreover, for $\rho < \alpha^2$, we can compute the
Hausdorff dimension of its range to be exactly $\max\left(d
\frac{\log 2}{\log (1/\alpha)},2\right)$.
To prove this notice that since $[0,1] \backslash \textsf{C}_{\rho,d}$ is a countable union of disjoint intervals, Hausdorff dimension of $(B-G)([0,1]\backslash \textsf{C}_{\rho,d})$ is equal to 2. Therefore it is enough to prove that $(B-G)(\textsf{C}_{\rho,d})$ has Hausdorff dimension equal to 
$d\frac{\log 2}{\log (1/\alpha)}$. Denote with $Q_n(p)$ the cube of side length
${\alpha^n}$, centered at $p$. It is easy to check that there is
a $\gamma>0$ such that, for all $t \in \textsf{C}_{\rho,d}$, the
cubes $Q_{n+1}(B_t-G_{n+1}(r))$ for $r \in D_n(t)$ will be
disjoint and contained in the cube $Q_n(B_t-G_n(t))$, even if
each of the smaller cubes is shifted by no more than $\gamma
\alpha^{n+1}$. Now knowing the modulus of continuity of Brownian
motion and using the same argument as in the proof of Theorem
\ref{thm:main} (and the fact that $\rho < \alpha^2$) we see that
for some (random) $n_0$ and all $n \geq n_0$ cubes
$Q_{n+1}(B_r-G_{n+1}(r))$ for $r \in D_n(t)$ will be disjoint
and contained in the cube $Q_n(B_t-G_n(t))$, even after being
shifted by no more than $\gamma \alpha^{n+1}/2$. As a
consequence, for any $m \geq n_0$ and $r \in D_m(t)$ every two
of the points $B_r-G_{m}(r)$ are at the distance greater than
$\gamma \alpha^m /2$. This immediately implies that the
Hausdorff dimension of the image $(B-G)(\textsf{C}_{\rho,d})$ is
almost surely bounded by $d \frac{\log 2}{\log (1/\alpha)}$.
For the lower bound we can use the Mass Distribution
Principle (see 4.2, Chapter 4 in \cite{Falconer}). For a fixed
$t \in \textsf{C}_{\rho,d}$ we define a measure on the part of
$(B-G)(\textsf{C}_{\rho,d})$ which is contained in
$Q_{n_0}(B_t-G_{n_0}(t))$. The measure is simply defined by
assigning mass $2^{-kd}$ to each of the disjoint cubes
$Q_{n_0+k}(B_s-G_{n_0+k}(s))$ for $s \in D_{n_0}(t)$. Then any
ball of radius $\gamma \alpha^{n_0+m}/2$ is assigned no more
than $2^{-md}$ mass and the claim follows by the Mass
Distribution Principle.

Therefore, for any $2<\lambda<d$ and any $\kappa<1/\lambda$, we
can find a generalized $d$-dimensional Hilbert curve which is
$\kappa$-H\"{o}lder continuous and such that almost surely the
range of Brownian motion with this drift has Hausdorff dimension
equal to $\lambda$, by taking $\alpha=\rho^{\kappa} =
2^{-d/\lambda}$ in our construction.

\end{remark}

\section{Proof of Theorem \ref{thm:standard Hilbert}}
We will prove Theorem \ref{thm:standard Hilbert} by showing that for $d>2$ the range of
$d$-dimensional Brownian motion with drift given by the standard
$d$-dimensional Hilbert curve hits points with positive probability.
The key ingredient in the proof is  the following lemma which shows that, for any $\alpha>1/d$, the Hilbert curve satisfies the reverse H\"{o}lder inequality for all points from a set of large measure.

\begin{lemma}\label{lemma: Holder}
For $d\ge 3$, let $H:[0,1]\to [0,1]^d$ denote the standard
$d$-dimensional Hilbert curve. Then for any $\alpha>1/d$ and any
$\epsilon>0$, there is a set $S \subset [0,1]$ of Lebesgue
measure greater than $1-\epsilon$ and a $C>0$ such that
$|H(s)-H(t)|\geq C|s-t|^\alpha$ for all $s,t \in S$.
\end{lemma}

\begin{proof}
As usual $\mathcal{F}_n$ is the set of dyadic cubes used in the
construction of the Hilbert curve. Define $W_n := \bigcup_{Q \in
\mathcal{F}_n}\partial Q$. For $\delta>0$ we define
$B(W_n,\delta)$ as the set of points in $[0,1]^d$ at distance no more than
$\delta$ from $W_n$.

First we show that the volume of the set $B(W_n,2^{-n-m})$ is
bounded from above by $2d2^{-m}$. We prove this by induction.
For $n=0$, this is trivial since $W_0$ is just the boundary of
the original cube whose ($d-1$)-dimensional volume is equal to
$2d$. When $m=0$, the claim is also trivial since
$B(W_n,2^{-n-m})$ is equal to the whole original cube. In the
general case, notice that the intersection of $B(W_n,2^{-n-m})$
with a cube from $\mathcal{F}_1$ is equal to
$B(W_{n-1},2^{-n-m+1})$, scaled by a factor of $1/2$ and
translated accordingly. The induction assumption implies that
the volume of this intersection is bounded from above by
$2d2^{-d} 2^{-m}$. Summing over all cubes in $\mathcal{F}_1$
proves the claim.

For $c>0$ define $K:=\bigcup_{n\geq 0}B(W_n, c2^{-nd\alpha })$ and set
$S:=H^{-1}(K^c)$. Since the volume of $K$ is less than $c\sum_{n
\geq 0}2^{-n(d\alpha-1)}$ and $H$ is measure preserving, we can
choose the constant $c$ to make the Lebesgue measure of $S$
arbitrarily close to $1$.

To check that $S$ satisfies the desired property, take $s,t \in
S$ and let $n$ be the largest integer so that $H(s)$ and $H(t)$
are contained in the same box in $\mathcal{F}_{n-1}$. Then
$|s-t| \leq 2^{-d(n-1)}$, but $|H(s)-H(t)| \geq
2c2^{-nd\alpha}$, since $H(s)$ and $H(t)$ are in $K^c$ and
inside different boxes in $\mathcal{F}_n$. This proves the claim with $C =
2^{1-d\alpha}c$.
\end{proof}


\begin{proof}[Proof of Theorem \ref{thm:standard Hilbert}]
Let $1/d<\alpha < \frac{3}{2(d+1)} \leq 1/2$ be given and let $H$ denote the regular
$d$-dimensional Hilbert curve. Let $J$ be a set of positive
measure with $|H(s)-H(t)| \geq C|s-t|^\alpha$ for all $s,t \in
J$, for some constant $C>0$, as in Lemma \ref{lemma: Holder}.
Without loss of generality we can assume $J$ is closed.

Fix $\delta>0$ and $x \in \mathbb{R}^d$, and set
$J_\delta=[\delta,1]\cap J$. We let $T_\epsilon$ be the amount
of time in $J_\delta$ that the range of function $t\mapsto B_t-H(t)$
spends in $\B:=B(x,\epsilon)$, the ball of radius $\epsilon$
around $x$. We will apply the second moment method to
$T_\epsilon$. Using Fubini and the fact that the density of
$d$-dimensional Brownian motion is bounded from below on $\B$ when $t\ge
\delta$, we have
\begin{equation}\label{eq: first moment}
\mathbb{E}(T_\epsilon) = \int_{J_\delta} \mathbb{P}(B_t-H(t)
\in \B)dt \geq b\epsilon^d,
\end{equation}
for some constant $b$ depending only on $\delta$ and $d$. To
find the second moment of $T_\epsilon$, apply Fubini to obtain
\begin{equation}\label{eq: second moment}
\mathbb{E}(T_\epsilon^2) = 2\int_{J_\delta}\mathbb{P}(B_s-H(s) \in \B)\int_{J_\delta\cap
[s,1]} \mathbb{P}(B_t-H(t) \in
\B|B_s-H(s) \in \B )dtds.
\end{equation}
By bounding density for $s \in J_\delta$ we have $\mathbb{P}(B_s-H(s) \in \B) \leq C_1
v_d\epsilon^d$, where $C_1$ depends on $\delta$ and $v_d$
denotes the volume of the unit ball in $d$ dimensions. Thus we
only need to estimate the inner integral in (\ref{eq: second moment}).
Define the set
$$
\tilde{J}_s=\{t \in J:t\geq s, |H(t)-H(s)| > 3\epsilon\}.
$$
The inner integral in (\ref{eq: second moment}) can be separated
into three parts. Let
\begin{align*}
& A_1:=\int_{[s,1]\cap (J_\delta \backslash \tilde{J}_s)}\mathbb{P}(B_t-H(t) \in \B |B_s-H(s) \in \B )dt \\
& A_2:=\int_{[s+(4\epsilon/C)^{1/\alpha},1] \cap J_\delta}\mathbb{P}(B_t-H(t) \in \B|B_s-H(s) \in \B )dt \\
& A_3:=\int_{[s,s+(4\epsilon/C)^{1/\alpha}] \cap J_\delta \cap
\tilde{J}_s}\mathbb{P}(B_t-H(t) \in \B|B_s-H(s) \in \B )dt.
\end{align*}
Because $H$ is measure preserving, the Lebesgue measure of
$J\backslash \tilde{J}_s$ can be bounded above by $v_d
3^d\epsilon^d$.
Thus, using a trivial bound for the probability
inside the integral, we have
\[A_1 \leq v_d3^d\epsilon^d.\]
For the second integral, notice that the distance between the
balls $B(x+H(s),\epsilon)$ and $B(x+H(t),\epsilon)$ is at least
$C(t-s)^\alpha -2\epsilon \geq C(t-s)^\alpha/2$. Therefore, by
bounding the density,
$$
A_2 \leq \int_{s+(4\epsilon/C)^{1/\alpha}}^1 v_d\epsilon^d
\,\frac{\exp\{-C^2(t-s)^{2\alpha-1}/8\}}{(2\pi(t-s))^{d/2}} dt
< C_2 v_d\epsilon^d,
$$
for some constant $C_2$, where we use the fact that $\lim_{r \downarrow 0}\exp (-r^p)r^q < \infty$ for $p < 0$ and any $q$.
Finally, if both $B_s-H(s)$ and $B_t-H(t)$ are in $\B$ and $t
\in \tilde{J}_s$, then it must hold that $|B_t-B_s| \geq
\epsilon$ and so
\begin{multline*}
A_3 \leq \int_s^{s+(4\epsilon/C)^{1/\alpha}}
\mathbb{P}\left(\frac{|B_t-B_s|}{\sqrt{t-s}} \geq
\frac{\epsilon}{\sqrt{t-s}}\right)dt \\
\leq \int_s^{s+(4\epsilon/C)^{1/\alpha}} C'\frac{\sqrt{t-s}}{\epsilon} dt =
\frac{2C'}{3}\Big(\frac{4}{C}\Big)^{3/(2\alpha)}\epsilon^{3/(2\alpha)-1},
\end{multline*}
where the second inequality follows by Markov inequality.
Since $\alpha < \frac{3}{2(d+1)}$, it follows that $A_3\le
C_3 \epsilon^d,$ for some constant $C_3$.

Now summing all the bounds, we have that $
\mathbb{E}(T_\epsilon^2) \leq K\epsilon^{2d} $ , for some $K>0$
and by Paley-Zygmund inequality we obtain
$$
\mathbb{P}(B_t-H(t) =x, \text{ for some } t) \geq \lim_{\epsilon
\to 0}\mathbb{P}(T_\epsilon > 0) \geq
\lim_{\epsilon
\to 0}\frac{\mathbb{E}(T_\epsilon)^2}{\mathbb{E}(T_\epsilon^2)} \geq
\frac{b^2}{K}.$$




\end{proof}


Theorem \ref{thm:standard Hilbert} tells us that even after the Brownian perturbation the standard Hilbert curve has positive volume. It would be interesting to see whether it also remains space filling.

We end the paper by giving a proof, different from the one in \cite{LeGall}, of Le~Gall's result that any function satisfying (\ref{d condition}) is polar. It is based on the fact that images of such functions have zero volume.
In the following $|S|$ will denote the diameter of a set $S$.
If $f \colon \mathbb{R}^+ \to \mathbb{R}^d$ satisfies (\ref{d condition}) then for any $\epsilon > 0$ we can find a $\delta > 0$ such that for any set $S \subset \mathbb{R}$ of diameter $|S| \leq \delta$ we have $|f(S)| \leq \epsilon |S|^{1/d}$. Covering an interval $I$ of length $l$ by intervals $I_1, \dots I_n$ of length $l/n$ for $n$ large enough, we obtain a cover $f(I_j)$, $1 \leq j \leq n$ for the image $f(I)$ which satisfies $\sum_{j}|f(I_j)|^d \leq \epsilon l$. Therefore the $d$-dimensional Hausdorff measure
of $f(I)$ is zero.
Since Brownian motion is almost surely $\alpha$-H\"{o}lder continuous for any $\alpha<1/2$ we have that $B-f$ satisfies the same condition (\ref{d condition}) and therefore its image has volume zero almost surely. Now the result follows by an application of Fubini and standard arguments.

\end{document}